\documentclass[amssymb,amscd,11pt,20pt]{amsart}
\usepackage{amsmath,amsthm}
\usepackage{amsfonts}
\usepackage{amscd}
\usepackage{amscd}
\usepackage[all]{xy}

 \newtheorem{thm}{Theorem}[section]

 \newtheorem{cor}[thm]{Corollary} 
 \newtheorem{lem}[thm]{Lemma} 
 
 \newtheorem{prop}[thm]{Proposition} 
  \theoremstyle{definition}
 \newtheorem{defn}[thm] {Definition} 
 \theoremstyle{remark}
 \newtheorem{rem}[thm]{Remark} 
 \newtheorem{ejem}[thm]{Example} 
  \newtheorem{ejems}[thm]{Examples}

\numberwithin{equation}{subsection}
\newcommand{\OO}{{\mathcal O}}
\newcommand{\LL}{{\mathcal L}}
\newcommand{\M}{{\mathcal M}}
\newcommand{\I}{{\mathcal I}}
\newcommand{\Pc}{{\mathcal P}}
\newcommand{\QQ}{{\mathcal Q}}

\newcommand{\U}{{\mathcal U}}
\newcommand{\V}{{\mathcal V}}
\newcommand{\Nc}{{\mathcal N}}

\newcommand{\ZZ}{{\mathbb Z}}

\newcommand{\RR}{\mathbb R}

\newcommand{\os}{\overset}
\newcommand{\Spec}{\operatorname{Spec}}
\newcommand{\id}{\operatorname{Id}}
\newcommand{\Qcoh}{\operatorname{Qcoh}}
\newcommand{\Mod}{\operatorname{Mod}}
\newcommand{\Qc}{\operatorname{Qc}}

\newcommand{\enumera}{\begin{enumerate}}
\newcommand{\eenumera}{\end{enumerate}}
\newcommand{\C}{{\mathcal C}}

\newcommand{\N}{{\mathcal N}}
 
\DeclareMathOperator{\Hom}{{Hom}}
\DeclareMathOperator{\di}{{d}}

\newcommand{\alineas}[1]{\begin{array}{#1}}
\newcommand{\alinea}{\begin{array}{l}}
\newcommand{\ealinea}{\end{array}}
\newcommand{\ealineas}{\end{array}}

\newcommand{\Ce}{\check{\mathcal{C}}}

\newsavebox\CBox

\begin{document}

\title{Derived category of Finite Spaces and Grothendieck Duality}

\author{  F. Sancho de Salas}\author{ J.F. Torres Sancho}

\address{ \newline Fernando Sancho de Salas\newline Departamento de
Matem\'aticas and Instituto Universitario de F\'isica Fundamental y Matem\'aticas (IUFFyM)\newline
Universidad de Salamanca\newline  Plaza de la Merced 1-4\\
37008 Salamanca\newline  Spain}
\email{fsancho@usal.es}

\address{ \newline Juan Francisco Torres Sancho\newline Departamento de
Matem\'aticas\newline
Universidad de Salamanca\newline  Plaza de la Merced 1-4\\
37008 Salamanca\newline  Spain}
\email{juanfran24@usal.es}

\subjclass[2010]{18E30, 06A11, 14F05}

\keywords{Finite spaces, quasi-coherent modules, Grothendieck duality, ringed spaces}

\thanks {The  authors were supported by research project MTM2017-86042-P (MEC)}




\begin{abstract} We obtain  some fundamental results, as Bokstedt-Neeman Theorem and Grothendieck duality, about the derived category of modules on a finite ringed space. Then we see how these results are transfered to  schemes in a simple way and generalized to other ringed spaces.  

\end{abstract}

\maketitle

\section*{Introduction}

Finite ringed spaces are a especially simple example of ringed space and they appear in a natural way as finite models of more general ringed spaces. In the topological context, the use of finite models for the study of general topological spaces goes back to   Mc Cord (\cite{McCord}) and it is still a useful tool nowadays  (see for example \cite{Barmak}, \cite{BarmakMinian} and \cite{BarmakMinian2}). In a more algebraic context, as the theory of schemes,   finite ringed spaces   (or more generally quivers with a representation) have been used for the study of the category of quasi-coherent modules on a scheme (for example, in  \cite{EnochsEstrada},\cite{Sancho}). Schemes admitting a finite model are precisely  quasi-compact and quasi-separated schemes. The finite models of these schemes  (resp. of quasi-compact and  semi-separated schemes) are an example of schematic finite spaces (resp. of semi-separated finite spaces), introduced in \cite{Sancho}, but there are schematic finite spaces that are not a finite model of a scheme. Our point of view is that a lot of concepts and results on schemes can be generalized to schematic finite spaces, recovering the results on schemes in  a more essential and - generally - simpler way, and allowing a further generalization to other ringed spaces. For example, the concept of  affine scheme and its different characterizations (as Serre's  cohomological criterion of affineness) led in \cite{Sancho} and \cite{Sanchos} to an analysis of the concept of affineness in the context of finite ringed spaces, schematic finite spaces and then to arbitrary ringed spaces.

 In this paper we continue this point of view: we obtain some fundamental results concerning the derived category of modules on a finite ringed space, recovering then the analogous results about the derived category of modules on a quasi-compact and quasi-separated scheme and then obtaining new results in other ringed spaces. More specifically: 


%

Let $(X,\OO)$ be a semi-separated finite space, $\Qcoh(X)$ the category of quasi-coherent $\OO$-modules and $D\Qcoh(X)$ its derived category. Let us denote by $D(X)$ the derived category of complexes of $\OO$-modules and by $D_{qc}(X)$ the full subcategory  of complexes of $\OO$-modules with quasi-coherent cohomology. Then we prove:

\medskip
\noindent{\bf\ \ Theorem \ref{DQc=D_qc}} (Bokstedt-Neeman theorem for semi-separated finite spaces). The natural functor $D\Qcoh(X)\to D_{qc}(X)$ is an equivalence.
\medskip

\noindent{\bf\ \ Theorem  \ref{Rqc=R}}. Let $f \colon X \to Y$ be a schematic morphism between semi-separated finite spaces.  The diagram 
\[ \xymatrix{  D\Qcoh(X)  \ar[r]^{\RR_{\text{qc}} f_*}  \ar[d] &  D\Qcoh(Y)   \ar[d] \\
D(X) \ar[r]^{\RR f_*} &   D(Y) }\]	is conmutative, where $\RR_{\text{qc}} f_*$ is the right derived functor of $f_*\colon \Qcoh(X)\to\Qcoh(Y)$. 
\medskip

\noindent{\bf\ \ Theorem \ref{enoughflats}}. The category $\Qcoh(X)$ has enough flats: any quasi-coherent module $\M$ admits a resolution 
\[ \cdots\to\Pc^{-2}\to\Pc^{-1}\to\Pc^0\to\M\to 0\] by quasi-coherent and flat $\OO$-modules.

\medskip

Regarding Grothendieck duality we shall first prove the following theorem for any morphism between  finite ringed spaces:

\medskip
\noindent{\bf\ \ Theorem \ref{generalduality}}. Let $f\colon X\to Y$ be a morphism between finite ringed spaces. The functor
$ \RR f_*\colon D(X)\to D(Y)$ has a right adjoint.
\medskip

This theorem generalizes the duality theorem of \cite{Navarro} from finite topological spaces  to arbitrary finite ringed spaces. In the quasi-coherent context we shall obtain:

\medskip
\noindent{\bf\ \ Theorem \ref{SchematicDuality}}. Let $f\colon X\to Y$ be a morphism between finite spaces (Definition \ref{finitespace}). One has:
\begin{enumerate}
\item The functor $\RR f_*\colon D\Qcoh (X)\to D(Y)$, composition of $D\Qcoh(X)\to D(X)$ and $  D(X)\overset{\RR f_*}\to D(Y)$,  has a right adjoint.

\item If $f$ is schematic, the functor $\RR f_* \colon D\Qcoh(X)\to D_{qc}(Y)$ has a right adjoint.

\item If $X$  and $f$ are schematic, the functor $\RR f_*\colon D_{qc}(X)\to D_{qc}(Y)$ has a right adjoint.

\item If $f$ is  schematic and $X, Y$ are  semi-separated,  the functor $\RR_{qc} f_*\colon D\Qcoh (X)\to D\Qcoh(Y)$ has a right adjoint.
\end{enumerate}
\medskip

In order to obtain all these results for schemes  and morphisms of schemes (in Theorem \ref{schemes}), we shall prove:

\medskip
\noindent{\bf\ \ Theorem  \ref{comparison2}}.   Let $S$ be a quasi-compact and quasi-separated  scheme and $\pi\colon S\to X$   a finite model. For any $\M\in D_{qc}(S)$, $\RR\pi_*\M$ belongs to $D_{qc}(X)$ and the functors 
\[ \RR\pi_*\colon D_{qc} (S)\to D_{qc}(X),\qquad \pi^*\colon D_{qc}(X)\to D_{qc} (S)\] are mutually inverse.
\medskip


Finally, let us see how to obtain new results for general ringed spaces. The notion of  affine ringed space was introduced in \cite{Sancho} (see also \cite{Sanchos}). This notion includes homotopically trivial topological spaces on the one side and affine schemes on the other side. Here we introduce the notion of  quasi-compact and quasi-separated ringed space (Definition \ref{qcqs}). In the topological case we obtain topological spaces that admit a finite model, as finite simplical complexes and $h$-regular finite CW-complexes. In the context of schemes, we recover the notion of a quasi-compact and quasi-separated scheme. Using the results of \cite{EnochsEstrada} we shall obtain (Theorem \ref{Thm-qcqs}) that the category $\Qcoh(S)$ of quasi-coherent modules on a quasi-compact and quasi-separated ringed space is a Grothendieck abelian category and   admits flat covers and cotorsion envelopes. Regarding Grothendieck duality,  we introduce the notion of a semi-separable ringed space, which  essentially means those ringed spaces that admit a semi-separated finite model and then we   give a Grothendieck duality theorem for quasi-coherent modules for a morphism (with certain ``schematic'' conditions) between semi-separable ringed spaces; the precise statement is:  

\medskip
\noindent{\bf \ \ Theorem \ref{dualityringedspaces}}. Let $f\colon T\to S$ be a morphism of ringed spaces. Assume that there exist an open covering $\U=\{U_1, \dots, U_n\}$   of $T$ and an open covering $\V=\{V_1,\dots ,V_m\}$ of $S$ satisfying: (for each $t\in T,s\in S$, we shall denote $U^t:=\underset{t\in U_i}\cap U_i$,   $V^s:=\underset{s\in V_j}\cap V_j$) 

(1) $U^t$ (resp. $V^s$) is affine, for any $t\in T$ (resp. $s\in S$).

(2) For any $U^t\supseteq U^{t'}$ (resp. $V^s\supseteq V^{s'}$), the morphisms 
\[ \OO_T(U^t)\to\OO_T(U^{t'}) \qquad (\text{resp.} \OO_S(V^s)\to\OO_S(V^{s'})) \] are flat.

(3) $H^i(U^t\cap U^{t'},\OO_T)=0$ for any $t,t'\in T$ and any $i>0$ (and analogolously for $S,\V$).

(4) $\OO_T(U^t\cap U^q )\otimes_{\OO_T(U^t)}\OO_T(U^{t'})\to \OO_T(U^{t'}\cap U^q )$ is an isomorphism for any $t,q\in T$ and any $U^t\supseteq U^{t'}$ (and analogously for $S,\V$).

(5) $U^t\subseteq f^{-1}(V^{f(t)})$ for any $t\in T$.

(6) For any $U^t\supseteq U^{t'}$, any $V^s\supseteq V^{s'}$ and any $i\geq 0$ the natural morphisms

\[ \aligned H^i(U^t\cap f^{-1}(V^s),\OO_T)\otimes_{\OO_T(U^t)}\OO_T(U^{t'}) &\to H^i(U^{t'}\cap f^{-1}(V^s),\OO_T) \\ 
H^i(U^t\cap f^{-1}(V^s),\OO_T)\otimes_{\OO_S(V^s)}\OO_S(V^{s'}) &\to H^i(U^{t}\cap f^{-1}(V^{s'}),\OO_T)
\endaligned
\] are isomorphisms. 

Then: the functor $\RR_{qc} f^{qc}_*\colon D\Qcoh(T)\to D\Qcoh (S)$ has a right adjoint, where $f^{qc}_*\colon \Qcoh(T)\to\Qcoh(S)$ is the right adjoint of $f^*\colon\Qcoh(S)\to\Qcoh(T)$, and $\RR_{qc} f^{qc}_*$ is its right derived functor.
\medskip

 This theorem may be understood as the essentialization of which properties of schemes are involved  in order to have a Grothendieck duality theorem for quasi-coherent modules. 

\section{Basics}

Let $X$ be a finite topological space. It is well known, since Alexandroff, that the topology of $X$ is equivalent to a preorder relation on $X$: $p\leq q$ iff $\bar p\subseteq \bar q$, where $\bar p,\bar q$ are the closures of $p$ and $q$. For each point $p\in X$ we shall denote   $$U_p=\text{ smallest open subset containing }p.$$ In other words $U_p=\{ q\in X: q\geq p\}$. Thus, $p\leq q\Leftrightarrow U_p\supseteq U_q$. A map $f\colon X\to Y$ between finite topological spaces is continuous if and only if it is monotone, i.e., $p\leq q$ implies $f(p)\leq f(q)$.

\begin{defn} A {\it finite ringed space} is a ringed space $(X,\OO)$ whose underlying topological space $X$ is finite. The sheaf of rings $\OO$ is always assumed to be a sheaf of commutative  rings with unity.
\end{defn}

A sheaf of rings $\OO$ on a finite topological space $X$ is equivalent to the following data: a ring $\OO_p$ for each $p\in X$ and a morphism of rings $r_{pq}\colon \OO_p\to\OO_q$ for each $p\leq q$, such that $r_{pp}=\id_{\OO_p}$ for any $p\in X$ and $r_{ql}\circ r_{pq}=r_{pl}$ for any $p\leq q\leq l$. One has that
\[ \OO_p=\text{\rm stalk of } \OO\,\text{\rm   at } p =\OO(U_p) \] and $r_{pq}$ is the restriction morphism $\OO (U_p)\to \OO(U_q)$.

A sheaf $\M$ of $\OO$-modules (or an $\OO$-module)  is equivalent to the following data: an $\OO_p$-module  $\M_p$ for each $p\in X$ and a morphism of $\OO_p$-modules $r_{pq}\colon \M_p\to\M_q$ for each $p\leq q$, such that $r_{pp}=\id_{\OO_p}$ for any $p\in X$ and $r_{ql}\circ r_{pq}=r_{pl}$ for any $p\leq q\leq l$. Again, one has that \[ \M_p=\text{\rm stalk at } p \,\text{\rm   of } \M =\M(U_p) \] and $r_{pq}$ is the restriction morphism $\M (U_p)\to \M(U_q)$. A morphism of $\OO$-modules $f\colon \M\to \N$ is equivalent to giving, for each $p\in X$, a morphism of $\OO_p$-modules  $f_p\colon\M_p\to\N_p$, which are compatible with the restriction morphisms $r_{pq}$.

If $\M$ is an $\OO$-module,  for each $p\leq q$ the morphism $r_{pq}$ induces a  morphism of $\OO_q$-modules $\widetilde r_{pq}\colon \M_p\otimes_{\OO_p}\OO_q\to\M_q$. It is proved in \cite{Sancho2} that $\M$ is a quasi-coherent $\OO$-module if and only if $\widetilde r_{pq}$ is an isomorphism for any $p\leq q$.

We shall denote by $\Mod(X)$ the category of $\OO$-modules on a ringed space $(X,\OO)$ and by $\Qcoh(X)$ the subcategory of quasi-coherent modules. For any ring $A$, $\Mod(A)$ denotes the category of $A$-modules.

\begin{ejem} The topological space with one element shall be denoted by $\{ *\}$. Thus, $(*,A)$ denotes the finite ringed space whose underlying topological space is $\{*\}$ and the sheaf of rings is a ring $\OO_*=A$. For any ringed space $(X,\OO)$ there is a natural morphism of ringed spaces $(X,\OO)\to (*,A)$, with $A=\OO(X)$.
\end{ejem}

\begin{defn}\label{finitespace} A {\it finite space} is a finite ringed space $(X,\OO)$ whose restriction morphisms $r_{pq}\colon \OO_p\to\OO_q$ are flat.
\end{defn}

The main properties of the category $\Qcoh(X)$  over a finite space are:

(1) $\Qcoh(X)$ is an abelian subcategory of $\Mod(X)$.

(2) $\Qcoh(X)$ is a Grothendieck category (see \cite{EnochsEstrada}).
\medskip

For any $\OO_p$-module $M$, we shall denote by $\widetilde M$ the quasi-coherent module on $U_p$ defined by $\widetilde M_x=M\otimes_{\OO_p}\OO_x$. In other words, $\widetilde M=\pi^*M$, where $\pi\colon (U_p,\OO_{\vert U_p})\to (*,\OO_p)$ is the natural morphism of ringed spaces. The functors
\[ \aligned \Qcoh(U_p) &\to \Mod(\OO_p)\\ \M &\leadsto \M_p\endaligned, \qquad \aligned \Mod(\OO_p) &\to \Qcoh(U_p)\\ M &\leadsto \widetilde M \endaligned\]  are mutually inverse. 

\begin{defn} A {\it schematic} finite space is a finite  space $(X,\OO)$ such that $R^i\delta_*\OO$ is quasi-coherent for any $i\geq 0$, where $\delta\colon X\to X\times X$ is the diagonal morphism. If in addition $R^i\delta_*\OO=0$ for any $i>0$, then we say that $(X,\OO)$ is a {\it semi-separated} finite space. 
\end{defn}

%
%

\begin{defn} A morphism $f\colon X\to Y$ is said to be {\it schematic} if  $R^i\Gamma_*\OO_X$ is quasi-coherent for any $i\geq 0$, where $\Gamma\colon X\to X\times Y$ is the graphic of $f$. For any schematic morphism $f\colon X\to Y$, the spaces $X$ and $Y$ are always assumed to be finite spaces (Definition \ref{finitespace}). 
\end{defn}

%

The following basic properties of schematic spaces, semi-separated spaces and schematic morphisms may be found in \cite{Sancho}.

\begin{prop}\label{basics}\begin{enumerate}

\item A morphism $f\colon X\to Y$ is   schematic if and only if: for any $x\leq x'\in X$,   any $y\leq y'\in Y$, and any $i\geq 0$, the natural morphisms

\[\aligned H^i(U_x\cap f^{-1}(U_y),\OO_X)\otimes_{\OO_{X,x}}\OO_{X,x'} &\to H^i(U_{x'}\cap f^{-1}(U_y),\OO_X) \\
H^i(U_x\cap f^{-1}(U_y),\OO_X)\otimes_{\OO_{Y,y}}\OO_{Y,y'} &\to H^i(U_{x}\cap f^{-1}(U_{y'}),\OO_X)
\endaligned \] are isomorphisms. 

\item A finite space $(X,\OO)$ is semi-separated if and only if:

(a) $H^i(U_p\cap U_q,\OO)=0$ for any $p,q\in X$ and any $i>0$.

(b) $\OO(U_p\cap U_q)\otimes_{\OO_p}\OO_{p'}\to \OO (U_{p'}\cap U_q)$ is an isomorphism for any $p,q\in X$ and any $p'\geq p$.

\item A finite space $(X,\OO)$ is schematic if and only if $U_p$ is semi-separated for any $p\in X$.

\item If $f\colon X\to Y$ is an schematic morphism, then $R^if_*\M$ is quasi-coherent for any $i\geq 0$ and any quasi-coherent module $\M$ on $X$. Moreover, for any $x\in X$, the induced morphism $\bar f\colon U_x\to U_{f(x)}$ satisfies: $R^i\bar f_*\M =0$, for any $i>0$ and any quasi-coherent module $\M$ on $U_x$, so
\[ \bar f_*\colon \Qcoh(U_x)\to \Qcoh(U_{f(x)})\] is an exact functor.

\item If $X$ is schematic, then the inclusion $j\colon U\hookrightarrow X$ is schematic, for any open subset $U$.

\item If $X$ is semi-separated, then for any $p\in X$, the inclusion $j\colon U_p\hookrightarrow X$ satisfies:
\begin{enumerate}
\item $j_*\colon \Qcoh(U_p)\to\Qcoh(X)$ is an exact functor and $H^i(X,j_*\M)=0$ for any $i>0$ and any quasi-coherent module $\M$ on $U_p$.

\item For any $x\in X$ the morphism $\OO_p\to \OO(U_p\cap U_x)$ is flat and for any quasi-coherent module $\M$ on $U_p$ the natural morphism
\[ \M_p\otimes_{\OO_p}\OO(U_p\cap U_x)\to \M(U_p\cap U_x)\] is an isomorphism.
\end{enumerate}
\end{enumerate}
\end{prop}

To conclude this section, we shall use without further mention that any Grothendieck abelian category has enough $K$-injectives (\cite{AlonsoJeremiasSouto}).

\section{Standard resolution, pseudo-cech resolution and quasicoherentation}


Let $(X,\OO)$ be a  finite ringed space of dimension $n$ and let $\M$ be an $\OO$-module. 

\begin{defn} We say that an ordered chain $x_0 < \cdots < x_i$ of points of $X$ belongs to an open subset $U$ (denoted by $(x_0 < \cdots < x_i) \in U$)   if all the $x_i$ belong to $U$;  since $U$ is open, it suffices that $x_0 \in U$.
\end{defn}
			 
\begin{defn} 
The {\it standard complex of $\M$} is the complex of $\OO$-modules 
$$ \C^{\scriptscriptstyle\bullet} \M :=0 \to \C^0 \M \to \C^1 \M \to \cdots \to \C^n \M \to 0 \qquad (n=\dim X)$$
defined as follows: for each  open set $U$ of $X$,
$$(\C^i \M) (U)= \prod_{(x_0 < \cdots < x_i) \in U} \M_{x_i} $$
and the restriction morphisms $(\C^i \M)(U) \to (\C^{i} \M)(V)$ are the natural projections (the set of chains belonging to $U$ is the disjoint union of the set of chains belonging to $V$ and the set of chains belonging to $U$ but not belonging to $V$).

The differential $d \colon \C^i \M \to \C^{i+1} \M $ is defined as follows: for each element $s=(s_{x_0 < \cdots < x_i}) \in (\C^i \M)(U)$, the element $d s  \in (\C^{i+1} \M)(U)$ is given by the formula:
$$(d s)_{x_0 < \cdots < x_{i+1}} = \sum_{k=0}^i (-1)^k s_{x_0 < \cdots < \widehat{x_k} < \cdots < x_{i+1}} + (-1)^{i+1} \bar{s}_{x_0 < \cdots < x_i} \in \M_{x_{i+1}},$$
where the notation $\widehat{x_k}$  means that we omit the element $x_k$ and $\bar{s}_{x_0 < \cdots < x_i}$ is the image of the element $s_{x_0 < \cdots < x_i} \in \M_{x_i}$ by the restriction morphism $\M_{x_i} \to \M_{x_{i+1}}$.

One easily checks that $ d \circ d=0$. There is also a natural morphism $\M \to \C^0 \M$, which is injective.
\end{defn}

\begin{rem} A morphism of modules $\M\to\M'$ induces a morphism of modules $\C^i\M\to\C^i\M'$ and then a morphism of complexes $\C^{\scriptscriptstyle\bullet}\M\to \C^{\scriptscriptstyle\bullet}\M'$. It is clear that $\M\leadsto\C^i\M$ is an exact functor.
\end{rem}

\begin{rem} \label{alternative def. of standard} 

Let us denote $\beta^i X =\{(x_0, \dots, x_i) \in X \times \os{i}{\cdots} \times X: x_0 < \cdots < x_i \}$ with the discrete topology and let $\OO_{\beta^i X}$ be the   sheaf of rings on $\beta^iX$ defined by $$(\OO_{\beta^i X})_{x_0 < \cdots < x_i} = \OO_{x_i}.$$   We have two natural morphisms of ringed spaces $\pi_0, \pi_i \colon (\beta^i X, \OO_{\beta^i X}) \to (X, \OO)$, defined as $\pi_0(x_0 < \cdots < x_i)= x_0$ and $\pi_i(x_0 < \cdots < x_i)= x_i$ (and the obvious morphisms between the sheaves of rings). Then 
$$ \C^i \M= \pi_{0*}(\pi_i^*\M).$$
\end{rem}

\begin{defn} For each open subset $U$ of $X$, we shall denote $$\M_{U}:= j_*\M_{|U},$$ where $j\colon U \hookrightarrow X$ is the natural inclusion.
\end{defn}

\begin{defn} The {\it pseudo-Cech complex of $\M$} is the complex of $\OO$-modules $$ \Ce^{\scriptscriptstyle\bullet} \M :=0 \to \Ce^0 \M \  {\to} \ \Ce^1 \M \ {\to} \cdots \to \Ce^n \M \to 0\qquad (n=\dim X)$$
defined by
$$\Ce^i \M = \prod_{x_0 < \cdots < x_i} \M_{U_{x_i}} $$
and the differential $\check{d} \colon \Ce^i \M \to \Ce^{i+1} \M $ is defined as follows: for each element $s=(s_{x_0 < \cdots < x_i}) \in (\Ce^i \M)(U)$, the element $\check{d} s  \in (\Ce^{i+1}\M)(U)$ is given by the formula:
$$(\check{d} s)_{x_0 < \cdots < x_{i+1}} = \sum_{k=0}^i (-1)^k s_{x_0 < \cdots < \widehat{x_k} < \cdots < x_{i+1}} + (-1)^{i+1} \bar{s}_{x_0 < \cdots < x_i} \in \M_{U_{x_{i+1}}}(U),$$
where the notation $\widehat{x_k}$  means  we omit the element $x_k$ and $\bar{s}_{x_0 < \cdots < x_i}$ is the image of the element $s_{x_0 < \cdots < x_i} \in \M_{U_{x_i}}(U)$ by the natural morphism $\M_{U_{x_i}}(U) \longrightarrow \M_{U_{x_{i+1}}}(U)$.

One easily checks that $ \check{d} \circ \check{d} =0$.  There is a natural morphism $\M \to \check{\C}^0 \M$ which is inyective. A morphism of modules $\M\to\M'$ induces a morphism of modules $\Ce^i\M\to\Ce^i\M'$ and then a morphism of complexes $\Ce^{\scriptscriptstyle\bullet}\M\to \Ce^{\scriptscriptstyle\bullet}\M'$.
\end{defn}

%
%

\subsection{Quasicoherentation}

Let $(X, \OO)$ be a finite space 
The inclusion $\Qcoh(X)\hookrightarrow \Mod(X)$ commutes with direct limits. Since $\Qcoh(X)$ is a Grothendieck category, it has a right adjoint
$$ \Qc \colon   \operatorname{Mod}(X) \to \Qcoh (X).$$

%
%
%
%
%
%

\begin{rem} One easily checks that $\Qc$ is   additive and left exact. Its restriction  to cuasicoherent $\OO$-modules is the identity.  
\end{rem}

The right derived functor $\RR\Qc\colon D(X)\to D\Qcoh(X)$ is a right adjoint of the natural functor $D\Qcoh(X)\to D(X)$.

\begin{prop}  \label{prop properties} Let $(X,\OO)$ be a finite ringed space and let   $\M$, $\N$ be $\OO$-modules.  

\begin{enumerate}
\item  For any $i\geq 0$, $$\Hom_{\OO}(\N , \Ce^i \M)= \Gamma(X, \C^i\underline{\Hom}_{\OO}(\N, \M)).$$    
\item  For any $i\geq 0$,   
$$ \Hom_{\OO}(\N, \C^i\M)= \prod_{x_0 < \cdots < x_i} \Hom_{\OO_{x_i}} ( \N_{x_0} \otimes_{\OO_{x_0}} \OO_{x_i},\M_{x_i}).$$

\item[(2')] If $\N$ is  quasicoherent, then
$$ \Hom_{\OO}(\N, \C^i\M)= \prod_{x_0 < \cdots < x_i} \Hom_{\OO_{x_i}} ( \N_{x_i},\M_{x_i})$$

\item If $(X, \OO)$ is  schematic, then
$$ \Qc(\C^i \M)= \prod_{x_0 < \cdots < x_i} j_* \widetilde{\M_{x_i}} $$
where  $ j \colon U_{x_i} \hookrightarrow X$ is the natural inclusion. 

\item[(3')] If $X$ is schematic and $\M$ is quasi-coherent, then 
\[ \Qc(\C^i \M)= \Ce^i \M.\]  
\item If $(X,\OO)$ is semi-separated, the functor $$\Qc \circ \C^i\colon \Mod(X)\to\Qcoh(X)$$ is   exact.
  
\item If $X$ is a finite space and $\M$ is   injective, then $\C^i\M$ is also injective.
\item If $X$ is semi-separated and $\M$ is   flat, then $\Qc(\C^i\M)$ is also flat.

\end{enumerate}

\end{prop}
\begin{proof}  (1) follows from the definitions and the isomorphism
\[ \Hom_\OO(\N,\M_{U_x})=\Hom_{\OO_{\vert U_x}}(\N_{\vert U_x}, \M_{\vert U_x})= \underline{\Hom}_{\OO}(\N, \M)_x\] for any $x\in X$. 

(2) follows from the equality $\C^i \M = \pi_{0*}(\pi_i^*\M) $ (Remark \ref{alternative def. of standard}), taking into account that $\beta^iX$ is discrete. If $\N$ is quasi-coherent, then  $\N_{x_0} \otimes_{\OO_{x_0}} \OO_{x_i}=\N_{x_i}$ and we obtain (2'). If $X$ is schematic and $\N$ is quasi-coherent, then $j_*\widetilde{\M_{x_i}}$ is quasi-coherent and $\Hom_{\OO_{x_i}} ( \N_{x_i},\M_{x_i})= \Hom_{\OO} ( \N,j_*\widetilde{\M_{x_i}})$. Thus (3) follows from (2'). Moreover, if $\M$ is quasi-coherent, then $j_*\widetilde{\M_{x_i}}=\M_{U_{x_i}}$ and we obtain (3').  If $X$ is semi-separated, then the functor $\M\leadsto j_*\widetilde \M_{x_i}$ is exact. Thus (4) follows from (3). 

(5) Let $\I$ be an injective $\OO$-module. In order to prove that $\C^i\I$ is injective, it suffices to see, by (2), that $\I_x$ is an injective $\OO_x$-module for any $x\in X$ (notice that the morphisms $\OO_{x_0}\to\OO_{x_i}$ are flat by hypothesis). For any $\OO_x$-module $N$, one has
\[ \Hom_{\OO_x}(N,\I_x)=\Hom_{\OO_{\vert U_x}}(\widetilde N, \I_{\vert U_x})\] and one concludes because $\I_{\vert U_x}$ is an injective $\OO_{\vert U_x}$-module.

(6) If $\M$ is flat, then $\M_x$ is a flat $\OO_x$-module for any $x\in X$ and then $\widetilde{\M_x}$ is a flat $\OO_{\vert U_x}$-module. We conclude by    (3)  and the following

\begin{lem} Let $X$ be a semi-separated finite space, $j\colon U_x\hookrightarrow X$ the natural inclusion and $\Pc$ a quasi-coherent and flat module on $U_x$. Then $j_*\Pc$ is  flat.
\end{lem}

\begin{proof} For any $p\in X$, one has that $(j_*\Pc)_p=\Pc(U_x\cap U_p)$. Since $X$ is semi-separated, one has an isomorphism $\Pc_x\otimes_{\OO_x}\OO (U_x\cap U_p)\overset\sim \to \Pc(U_x\cap U_p)$ which is a flat $\OO_p$-module because $\Pc_x$ is a flat $\OO_x$-module and $\OO_p\to \OO(U_x\cap U_p)$ is flat.
\end{proof}
\end{proof}

\begin{thm}\label{resolutions} Let $(X,\OO)$ be a finite ringed space,  $\M$ an $\OO$-module.

\begin{enumerate}

\item $\C^{\scriptscriptstyle\bullet} \M$ is a finite and  flasque resolution of $\M$. 

\item If $(X,\OO)$ is  semi-separated  and $\M$ is   quasi-coherent, then   $\Ce^{\scriptscriptstyle\bullet}\M$ is a resolution of $\M$ by acyclic quasi-coherent $\OO$-modules.

\end{enumerate}
\end{thm}

\begin{proof} 

(1) See \cite[Theorem 2.15]{Sancho}.

(2) $\Ce^i\M$ is acylic, since $\M_{U_x}$ is acyclic for any $x\in X$ because $X$ is semi-separated. Let us see that $\Ce^{\scriptscriptstyle\bullet}\M$ is a resolution of $\M$. For any open subset $U$  of $X$, let us denote $\OO^{U}$ the sheaf $\OO$ supported on $U$. For any $\OO$-module $\LL$, one has:
\[ \Hom_{\OO} (\OO^U, \LL)= \Gamma(U, \LL),\qquad 
 \underline{\Hom}_{\OO} (\OO^U,\LL)=  \LL_U. \]
Then $$ (\Ce^{\scriptscriptstyle\bullet} \M)_x =  \Hom_{\OO} (\OO^{U_x}, \Ce^{\scriptscriptstyle\bullet} \M)= \Gamma(X, \C^{\scriptscriptstyle\bullet}(\underline{\text{Hom}}_{\OO}(\OO^{U_x}, \M)))= \Gamma(X, \C^{\scriptscriptstyle\bullet}(\M_{U_x})),$$ where the second equality is due to Proposition \ref{prop properties}, (4). Thus, $H^i [(\Ce^{\scriptscriptstyle\bullet} \M)_x] = H^i(X, \M_{U_x}).$ Since $X$ is semi-separated, $H^i(X, \M_{U_x})=0$ for $i>0$ and we are done.
\end{proof}

\begin{rem} If $\M$ is a complex of $\OO$-modules, then $\C^{\scriptscriptstyle\bullet}\M$ denotes the simple (or total) complex associated to the bicomplex $\C^p\M^q$. Analogously, $\Ce^{\scriptscriptstyle\bullet}\M$ denotes the simple complex associated to the bicomplex $\Ce^p\M^q$. Taking into account the boundedness of the complexes $C^{\scriptscriptstyle\bullet}$ and $\Ce^{\scriptscriptstyle\bullet}$, Theorem \ref{resolutions} yields that $\M\to \C^{\scriptscriptstyle\bullet}\M$ is a quasi-isomorphism and so is $\M\to\Ce^{\scriptscriptstyle\bullet}\M$ if $X$ is semi-separated and $\M$ is a complex of quasi-coherent modules.
\end{rem}

\section{Semi-separated finite spaces}

Let $(X,\OO)$ be a finite space. We shall denote by $D(X)$ the derived category of complexes of $\OO$-modules and by $D\Qcoh(X)$ the derived category of complexes of quasi-coherent $\OO$-modules. We shall denote by $D_{qc}(X)$ the full subcategory of $D(X)$ whose objects are the complexes of $\OO$-modules with quasi-coherent cohomology. The objects of these categories have a very simple description:

A complex $\M$ of $\OO$-modules is the same as giving:   a complex $\M_x$ of $\OO_x$-modules, for each $x\in X$,  and a morphism of complexes of $\OO_x$-modules 
\[ r_{xx'}\colon \M_x\to \M_{x'}\] for each $x\leq x'$, such that $r_{xx}=\id$ for any $x$ and $r_{xx''}=r_{x'x''}\circ r_{xx'}$ for any $x\leq x'\leq x''$. If we denote $$\widetilde r_{xx'}\colon \M_x\otimes_{\OO_x}\OO_{x'}\to\M_{x'}$$ the  morphism of $\OO_{x'}$-modules induced by $r_{xx'}$ then:
\begin{enumerate}
\item $\M$ is a complex of quasi-coherent modules if and only if $\widetilde r_{xx'}$ is an isomophism for any $x\leq x'$.
\item  $\M$ is a complex with quasi-coherent cohomology if and only if $\widetilde r_{xx'}$ is a  quasi-isomorphism   for any $x\leq x'$.
\end{enumerate} 

\begin{thm}\label{Qc-derived} Let $(X,\OO)$ be a  semi-separated finite space. Then $\C^i\M$ is  $\Qc$-acyclic, for any $\OO$-module $\M$ and any $i\geq 0$. Consequently, 
\begin{enumerate} \item $R^i\Qc(\M)=0$ for any $i>\operatorname{dim} X$ and any $\OO$-module $\M$.
 \item For any complex $\M$ of $\OO$-modules  
$$ \RR Qc(\M)\simeq Qc(\C^{\scriptscriptstyle\bullet} \M).$$ In particular, any complex $\M$ of quasi-coherent modules is $\Qc$-acylic and $\M\simeq \RR\Qc(\M)$.
\end{enumerate}
\end{thm}

\begin{proof} Let $\M\to\I^{\scriptscriptstyle\bullet}$ be an injective resolution. By  Proposition \ref{prop properties}, (4), $\Qc(\C^i\M)\to \Qc(\C^i\I^{\scriptscriptstyle\bullet})$ is still a resolution. We conclude because $\C^i\I^{\scriptscriptstyle\bullet}$ is a resolution of $\C^i\M$ ($\C^i$ is exact) by injectives (Proposition \ref{prop properties}, (5)).

If $\M$ is a complex of quasi-coherent modules, then $\Qc(\C^{\scriptscriptstyle\bullet}\M)\overset{\ref{prop properties}, \text{(3')}}=\Ce^{\scriptscriptstyle\bullet}\M \overset{\ref{resolutions}}\simeq \M$.

\end{proof}

\begin{thm}[Bokstedt-Neeman Theorem for semi-separated finite spaces]\label{DQc=D_qc} Let $(X,\OO)$ be a  semi-separated finite space. The functor $D\Qcoh (X) \to D_{qc}(X)$ is an equivalence.
\end{thm}
\begin{proof} By Theorem \ref{Qc-derived}, $\RR\Qc(\M)\simeq \Qc(\C^{\scriptscriptstyle\bullet}\M)$. The key point is to prove that the natural morphism
\[ \Qc(\C^{\scriptscriptstyle\bullet}\M)\to \C^{\scriptscriptstyle\bullet}\M\] is a quasi-isomorphism for any complex $\M$ with quasi-coherent cohomology. It suffices to prove that
\[ H^i_{\di_\M}[\Qc(\C^{\scriptscriptstyle\bullet}\M)]\to H^i_{\di_\M}[\C^{\scriptscriptstyle\bullet}\M]\] is a quasi-isomorphism, where $\di_\M\colon \C^p\M^q\to \C^p\M^{q+1}$ is the ``vertical'' differential of the bicomplex $\C^p\M^q$ (and analogously for the bicomplex $\Qc(\C^p\M^q)$). Since $\Qc\circ \C^{\scriptscriptstyle\bullet}$ and $\C^{\scriptscriptstyle\bullet}$ are exact, this amounts to prove that 
\[ \Qc(\C^{\scriptscriptstyle\bullet}H^i(\M))\to \C^{\scriptscriptstyle\bullet}H^i(\M)\] is a quasi-isomorphism; that is, we may assume that  $\M$ is a quasi-coherent module. In this case, $\Qc(C^{\scriptscriptstyle\bullet}\M)=\Ce^{\scriptscriptstyle\bullet}\M$ by Proposition \ref{prop properties}, (3'). We conclude because $\M\to \Ce^{\scriptscriptstyle\bullet}\M$ and $\M\to C^{\scriptscriptstyle\bullet}\M$ are quasi-isomorphisms (Theorem \ref{resolutions}).

%

Now let us conclude the proof of the theorem. If $\M\in D_{qc(X)}$, the quasi-isomorphism $\Qc(\C^{\scriptscriptstyle\bullet}\M)\to \C^{\scriptscriptstyle\bullet}\M$ gives an isomorphism in $D_{qc}(X)$, $\RR\Qc(\M)\simeq \M$, i.e., the composition $D_{qc}(X)\overset{\RR\Qc}\to D\Qcoh(X)\to D_{qc}(X)$ is isomorphic to the identity. If $\M\in D\Qcoh(X)$, the natural morphism
$\M=\Qc(\M)\to \Qc(\C^{\scriptscriptstyle\bullet}\M)$ is a quasi-isomorphism, because its composition with $\Qc(\C^{\scriptscriptstyle\bullet}\M)\to \C^{\scriptscriptstyle\bullet}\M$ is a quasi-isomorphism. Thus, we have obtained an isomorphism  $\M\simeq \RR\Qc(\M)$ in $D\Qcoh(X)$, so the composition $D\Qcoh(X)\to D_{qc}(X)\to D\Qcoh(X)$ is isomorphic to the identity.
\end{proof}

\begin{cor}\label{DqcInD} Let $X$ be an schematic finite space. The functor $D_{qc}(X)\hookrightarrow D(X)$ has a right adjoint.
\end{cor}

\begin{proof} We have to prove:

(*)  for any $\Nc\in D(X)$ there exists an object $\Nc_{qc}\in D_{qc}$ and a morphism $\Nc_{qc}\to\Nc$ such that $\Hom_{D(X)}(\M,\Nc_{qc})\to\Hom_{D(X)}(\M,\Nc)$ is an isomorphism for any $\M\in D_{qc}(X)$. 

We proceed by induction on $\# X=$ number of elements of $X$. If $\# X=1$, then $X$ is semi-separated and then $\RR\Qc\colon D(X)\to D\Qcoh(X)\overset{\ref{DQc=D_qc}}\simeq D_{qc}(X)$ is a right adjoint. Now let $\# X$ be greater than 1. If $X$ has a minimun, $X=U_p$, then $X$ is semi-separated and we conclude as before. If $X$ has not a minimum, then $X=U\cup V$, with $U,V$ open subsets different from $X$. Let us denote $i\colon U\hookrightarrow X$, $j\colon V\hookrightarrow X$ and $h\colon U\cap V\hookrightarrow X$ the inclusions. By induction, there exist $(\Nc_{\vert U})_{qc}\in D_{qc}(U)$,  $(\Nc_{\vert V})_{qc}\in D_{qc}(V)$, $(\Nc_{\vert U\cap V})_{qc}\in D_{qc}(U\cap V)$ and morphisms $\alpha\colon (\Nc_{\vert U})_{qc}\to \Nc_{\vert U}$, $\beta\colon (\Nc_{\vert V})_{qc}\to \Nc_{\vert V}$ and $\gamma\colon (\Nc_{\vert U\cap V})_{qc}\to \Nc_{\vert U\cap V}$ satisfying (*). Then $\alpha_{\vert U\cap V}$ and $\beta_{\vert U\cap V}$ factor through $\gamma$. Hence we obtain morphisms
\[  \RR {i}_* (\Nc_{\vert U})_{qc} \overset\phi\longrightarrow \RR {h}_* (\Nc_{\vert U\cap V})_{qc},\qquad \RR {j}_* (\Nc_{\vert V})_{qc} \overset\psi\longrightarrow  \RR {h}_* (\Nc_{\vert U\cap V})_{qc}\] and commutative diagrams
\[ \xymatrix{ \RR {i}_* (\Nc_{\vert U})_{qc} \ar[r]\ar[d] &  \RR {h}_* (\Nc_{\vert U\cap V})_{qc}\ar[d] \\ \RR {i}_* \Nc_{\vert U} \ar[r] & \RR {h}_* \Nc_{\vert U\cap V} }\qquad \xymatrix{ \RR {j}_* (\Nc_{\vert V})_{qc} \ar[r]\ar[d] &  \RR {h}_* (\Nc_{\vert U\cap V})_{qc}\ar[d] \\ \RR {j}_* \Nc_{\vert V} \ar[r] & \RR {h}_* \Nc_{\vert U\cap V} }
\]
Let us consider the triangle $\Nc\to  \RR {i}_* \Nc_{\vert U}\oplus \RR {j}_* \Nc_{\vert V}\to \RR {h}_* \Nc_{\vert U\cap V}$ and the commutative diagram
\[\xymatrix{ \Nc \ar[r] &  \RR {i}_* \Nc_{\vert U}\oplus \RR {j}_* \Nc_{\vert V}\ar[r] & \RR {h}_* \Nc_{\vert U\cap V}\\ & \RR {i}_* (\Nc_{\vert U})_{qc} \oplus \RR {j}_* (\Nc_{\vert V})_{qc} \ar[r]^{\qquad\phi-\psi}\ar[u] & \RR {h}_* (\Nc_{\vert U\cap V})_{qc}\ar[u]
  }. 
\] Let us define $\Nc_{qc}:=\operatorname{Cone}(\phi-\psi)[-1]$ and let $\Nc_{qc}\to \Nc$ be a morphism completing the above diagram to a morphism of   exact triangles. This morphism satisfies (*). Indeed, for any $\M\in D_{qc}(X)$ one has that
\[ \Hom(\M,\RR i_*\Nc_{\vert U})= \Hom(\M_{\vert U}, \Nc_{\vert U}) =\Hom(\M_{\vert U}, (\Nc_{\vert U})_{qc})= \Hom(\M,\RR i_*(\Nc_{\vert U})_{qc})\] and analogously for $V$ and $U\cap V$. One concludes by applying $\Hom(\M,\quad)$ to the morphism  of exact triangles.
\end{proof}


\begin{rem}\label{remarkBN} There are non-schematic finite spaces where Bokstedt-Neeman theorem holds. For example, for any affine finite space $X$, one has an equivalence $D_{qc}(X)\simeq D\Qcoh(X)$ and they are both equivalent to $D(A)$, with $A=\OO(X)$ (see \cite[Proposition 3.17]{Sancho}). Thus, for any finite space $X$, Bokstedt-Neeman holds locally: $D_{qc}(U_x)\simeq D\Qcoh (U_x)\simeq D(\OO_x)$ for any $x\in X$.
\end{rem}

Let $f \colon X \to Y$ be a schematic morphism between finite spaces. Since $\Qcoh(X)$ is a Grothendieck abelian category, the functor $f_*\colon \Qcoh(X)\to\Qcoh (Y)$ has a right derived functor  
\[ \RR_{\text{qc}}f_*\colon D\Qcoh(X)\to D\Qcoh(Y).\]

%

\begin{thm}\label{Rqc=R} Let $f \colon X \to Y$ be a schematic morphism between semi-separated finite spaces.  The diagram 
\[ \xymatrix{  D\Qcoh(X)  \ar[r]^{\RR_{\text{qc}} f_*}  \ar[d] &  D\Qcoh(Y)   \ar[d] \\
D(X) \ar[r]^{\RR f_*} &   D(Y) }\]	is conmutative.		
\end{thm}
				
\begin{proof} The proof consists on proving that the pseudo-Cech resolution $\M \to \Ce^{{\scriptscriptstyle\bullet}} \M$ allows us to derive both functors, and this reduces	to prove that for any $  x \in X$ and any quasi-coherent module $\M$ on $X$ one has:							
\begin{enumerate}
\item  $\RR f_* \M_{U_x} = f_* \M_{U_x}$.
\smallskip

\item  $\RR_{qc} f_* \M_{U_x} = f_* \M_{U_x}$.
\end{enumerate} 			
  
Let us consider the conmutative  diagram 
$$
\xymatrix{ X \ar[r]^f & Y \\ U_x \ar@{^(->}[u]^i \ar[r]^{\bar{f} \ \ } & \  \ U_{f(x)} \ar@{^(->}[u]_j .} 
$$

Let us prove (1). One has: $\M_{U_x} = i_* \M_{|U_x}= \RR i_* \M_{|U_x}$, because $\M_{|U_x}$ is cuasicoherent and $X$ is semi-separated. Thus:
$$\RR f_*\M_{U_x}= \RR f_* \RR i_* \M_{|U_x}=  \RR j_*(\RR \bar{f}_*  \M_{|U_x}).$$  Now, $\bar{f}_* \M_{|U_x}=\RR \bar f_*\M_{|U_x}$ because $f$ is schematic (see Proposition \ref{basics}) and $ j_* \bar{f}_* \M_{|U_x}= \RR j_* \bar{f}_* \M_{|U_x}$ because $Y$ is semi-separated. Hence, $\RR f_*\M_{U_x}= j_* \bar{f}_* \M_{|U_x}= f_* \M_{U_x}$.

Now let us prove (2). Let $ \mathcal{I}^{\scriptscriptstyle\bullet}$ be a resolution of $\M_{|U_x}$ by  injective quasi-coherent  modules. Since $i_*\colon \Qcoh(U_x)\to\Qcoh(X)$ is exact and takes injectives into injectives, we have that $i_* \mathcal{I}^{\scriptscriptstyle\bullet}$  is a resolution  of $\M_{U_x}$ by quasi-coherent  injective $\OO_X$-modules. Then
$$ \RR_{qc} f_* \M_{U_x}= f_* i_* \mathcal{I}^{\scriptscriptstyle\bullet} = j_* \bar{f}_* \mathcal{I}^{\scriptscriptstyle\bullet}. $$		
Now,   $j_* \bar{f}_* \mathcal{I}^{\scriptscriptstyle\bullet}$ is a resolution of $j_* \bar{f}_* \M_{|U_x}=f_* \M_{U_x}$, because  $\bar f_*\colon \Qcoh(U_x)\to\Qcoh(U_{f(x)})$ and $j_*\colon \Qcoh(U_{f(x)})\to\Qcoh(Y)$ are exact functors. 
\end{proof}


%
%
%
%
%

\begin{thm}\label{enoughflats} The category of quasi-coherent modules on a semi-separated finite space has enough flats.
\end{thm}

\begin{proof} Let $\M$ be a quasi-coherent module on a semi-separated finite space $(X,\OO)$. Let $$\Pc= \cdots \to \Pc^{-1}\to \Pc^0\to   0$$ be a resolution of $\M$ by flat (non quasi-coherent) modules. Let $\QQ:=\Qc(\C^{\scriptscriptstyle\bullet}\Pc)$. By Proposition \ref{prop properties}, (6), $\QQ$ is a bounded above complex of flat quasi-coherent modules. Moreover, $\QQ\simeq\RR\Qc(\Pc)\simeq\RR\Qc(\M)\simeq \M$, hence $H^i(\QQ)=0$ for $i\neq 0$ and $H^0(\QQ)=\M$. Thus, 
\[ \tau_{\leq 0}\QQ= \cdots\to\QQ^{-2}\to\QQ^{-1}\to Z^0\to 0\] with $Z^0$ the 0-cycles of $\QQ$, is a resolution of $\M$ by quasi-coherent modules and we conclude if we prove that $Z^0$ is flat. This follows from the fact that
\[ 0\to Z^0\to \QQ^0\to\QQ^1\to\cdots\to\QQ^n\to 0\] is an exact sequence and $\QQ^i$ are flat.
\end{proof}


%
%
%
%

\section{Grothendieck duality}

\begin{thm}\label{generalduality} Let $f\colon X\to Y$ be a morphism between finite ringed spaces. The functor $\RR f_*\colon D(X)\to D(Y)$ has a right adjoint $f^!\colon D(Y)\to D(X)$.
\end{thm}

\begin{proof} Let $n=\dim X$. For each $0\leq p\leq n$ the functor $f_*\C^p\colon \Mod(X)\to\Mod(Y)$ is exact and conmmutes with filtered direct limits. Thus, for each $\Nc\in \Mod(Y)$ the functor
\[ \aligned \Mod(X)&\to \text{ Abelian groups }\\ \M &\leadsto \Hom(f_*\C^p(\M),\Nc)\endaligned\] is representable by an $\OO_X$-module $f^{-p}\Nc$. A morphism of $\OO_Y$-modules $\Nc\to \Nc'$ induces a morphism of $\OO_X$-modules $f^{-p}\Nc\to f^{-p}\Nc'$. Moreover, the natural morphism $f_*\C^p\to f_*\C^{p+1}$ induces a morphism $f^{-p-1}\Nc\to f^{-p}\Nc$. Thus, if $\N$ is a complex of $\OO_Y$-modules, we obtain a bicomplex $f^{-p}\Nc^q$ (whis is zero whenever $p<0$ or $p>n$) whose associated simple complex shall be denoted by $f^\nabla\Nc$. For any complex $\M$ of $\OO_X$-modules, the isomorphism
\[ \Hom_{\OO_Y}(f_*\C^p(\M^i),\Nc^q)=\Hom_{\OO_X}(\M^i,f^{-p}\Nc^q)\] extends to a complex isomorphism
\[ \Hom^{\scriptscriptstyle\bullet}(f_*\C^{\scriptscriptstyle\bullet}(\M),\Nc)=\Hom^{\scriptscriptstyle\bullet}(\M,f^\nabla\Nc).\] Thus, if $\Nc$ is $K$-injective, $f^\nabla\Nc$ is $K$-injective too. For any $\Nc\in D(Y)$, we define $f^!\Nc:= f^\nabla \I$, with $\Nc\to\I$ a $K$-injective resolution, and one has
\[ \Hom^{\scriptscriptstyle\bullet}(f_*\C^{\scriptscriptstyle\bullet}(\M),\I)=\Hom^{\scriptscriptstyle\bullet}(\M,f^\nabla\I) \] and then
\[ \RR\Hom^{\scriptscriptstyle\bullet}(\RR f_*(\M),\Nc)=\RR\Hom^{\scriptscriptstyle\bullet}(\M,f^!\Nc). \]

\end{proof}

\begin{thm}\label{SchematicDuality} Let $f\colon X\to Y$ be a morphism between finite spaces (Definition \ref{finitespace}). One has:
\begin{enumerate}
\item The functor $\RR f_*\colon D\Qcoh (X)\to D(Y)$, composition of $D\Qcoh(X)\to D(X)$ and $  D(X)\overset{\RR f_*}\to D(Y)$,  has a right adjoint.

\item If $f$ is schematic, the functor $\RR f_* \colon D\Qcoh(X)\to D_{qc}(Y)$ has a right adjoint.

\item If $X$  and $f$ are schematic, the functor $\RR f_*\colon D_{qc}(X)\to D_{qc}(Y)$ has a right adjoint.

\item If $f$ is  schematic and $X, Y$ are  semi-separated,  the functor $\RR_{qc} f_*\colon D\Qcoh (X)\to D\Qcoh(Y)$ has a right adjoint.
\end{enumerate}
\end{thm}

\begin{proof} (1) follows from Theorem \ref{generalduality} and from the fact that $D\Qcoh(X)\to D(X)$ has a right adjoint (in fact, $\RR\Qc$).

(2) Since $f$ is schematic, $\RR f_*$ maps  $D_{qc}(X)$ into $D_{qc}(Y)$. We conclude by (1).

(3) follows from Theorem \ref{generalduality} and Corollary \ref{DqcInD}.

(4) follows from (3) and Theorems  \ref{DQc=D_qc},  \ref{Rqc=R}.
\end{proof}

\section{Schemes}

Let $S$  be a quasi-compact and quasi-separated scheme. It is proved in \cite{Sancho2} that there exists a schematic finite space $(X,\OO)$ and a morphism of ringed spaces $$\pi\colon (S,\OO_S)\to (X,\OO)$$ such that, for any $x\in X$, the preimage $U^x:=\pi^{-1}(U_x)$ is an affine scheme. Moreover, $S$ is a semi-separated scheme if and only if $(X,\OO)$ is a semi-separated finite space. We say that $S\to X$ is a {\it finite model} of $S$. If $f\colon T\to S$ is a morphism of schemes between quasi-compact and quasi-separated schemes, one can find finite models $X$ and $Y$ of $T$ and $S$ and a (schematic) morphism $\bar f \colon X \to Y$ making the diagram
\[ \xymatrix{T\ar[r]^f\ar[d] & S\ar[d] \\ X \ar[r]^{\bar f } & Y
}\] commutative, and we say that $\bar f $ is a finite model of $f$.

\begin{thm}(\cite[Thm. 3.15]{Sancho2})\label{comparison} Let $S$ be a quasi-compact and quasi-separated scheme and let $\pi\colon S\to X$ be a finite model. The functors 
\[ \pi_*\colon \Qcoh(S)\to\Qcoh(X),\qquad \pi^*\colon\Qcoh(X)\to \Qcoh (S)\] are exact and mutually inverse. In particular, we obtain an equivalence $D\Qcoh(S)\simeq D\Qcoh(X)$.
\end{thm}

 Let us see now that the same happens with the derived categories of complexes with quasi-coherent cohomology. We shall need the following:

\begin{thm}{(\cite[Thm. 5.1]{BokstedtNeeman})}\label{BN} Let $S=\Spec A$ be an affine scheme and $\pi\colon \Spec A\to (*,A)$ the natural morphism. The functor
\[\pi^*\colon D (A) \to D_{qc}(S)\] is an equivalence.
\end{thm}

\begin{rem}\label{remarkBN2} \begin{enumerate} \item Since $\RR\pi_*\colon D_{qc}(S)\to D(A)$ is a right adjoint of $\pi^*$, it is its inverse. Thus the natural morphisms $M\to \RR\pi_*\pi^*M$ and $\pi^*\RR\pi_*\M\to\M$ are isomorphisms, for any $M\in D(A)$, $\M\in D_{qc}(S)$. Notice also that $\RR\pi_*=\RR\Gamma(S,\quad)$.
\item If $S'=\Spec A'$ is an affine open subscheme of $S$, then $j^*\pi^*M \simeq {\pi'}^*(M\otimes_A A')$, where $j\colon S'\hookrightarrow S$ is the natural immersion and $\pi'\colon S'\to (*,A')$ is the natural morphism. It follows that  the natural morphism
\[ \RR\Gamma(S,\M)\otimes_AA'\to\RR\Gamma(S',\M)\] is an isomorphism.
\end{enumerate}
\end{rem}

\begin{thm}\label{comparison2} Let $S$ be a quasi-compact and quasi-separated  scheme and $\pi\colon S\to X$   a finite model. For any $\M\in D_{qc}(S)$, $\RR\pi_*\M$ belongs to $D_{qc}(X)$ and the functors 
\[ \RR\pi_*\colon D_{qc} (S)\to D_{qc}(X),\qquad \pi^*\colon D_{qc}(X)\to D_{qc} (S)\] are mutually inverse.
\end{thm}

\begin{proof} Let $\M\in D_{qc}(S)$. In order to prove that $\RR\pi_*\M$ has quasi-coherent cohomology, we have to see that for any $x\leq x'$ the natural morphism
\[(\RR\pi_*\M)_x\otimes_{\OO_x}\OO_{x'}\to (\RR\pi_*\M)_{x'}\] is an isomorphism. Since $(\RR\pi_*\M)_x=\RR\Gamma(U^x,\M)$, this follows from Remark \ref{remarkBN2}. Now, for any $x\in X$, the functors
\[ \RR\pi_*\colon D_{qc}(U^x)\to D_{qc}(U_x),\quad  \pi^*\colon D_{qc}(U_x)\to D_{qc}(U^x)\] are mutually inverse, since both categories are equivalent to $D(\OO_x)$ by Theorems \ref{BN} and Remark \ref{remarkBN}. For any $\M\in D_{qc}(S)$, the natural morphism $\pi^*\RR\pi_*\M\to\M$ is an isomorphism, because it is so after restricting to each $U^x$. For any $\Nc\in D_{qc}(X)$ the natural morphism $\Nc\to\RR\pi_*\pi^*\Nc$ is an isomorphism because it is so after restricting to each $U_x$.
\end{proof}

\begin{rem}\label{remark} (1) Let $S$ be a quasi-compact and quasi-separated  scheme and $\pi\colon S\to X$   a finite model. The diagram
\[ \xymatrix{D\Qcoh(S)\ar[d]^{\pi_*}_{\wr}\ar[r] & D_{qc}(S)\ar[d]^{\RR \pi_*}_{\wr} \\ D\Qcoh(X)\ar[r]  & D_{qc}(X)
} \] (whose vertical morphisms are isomorphisms by Theorems \ref{comparison} and \ref{comparison2}) is commutative. Indeed, let us denote $i\colon D\Qcoh(S)\to D_{qc}(S)$ and $j\colon D\Qcoh(X)\to D_{qc}(X)$ the natural functors.  For any $\Nc\in D\Qcoh(X)$, one has $\pi^*j(\Nc)=i(\pi^*\Nc)$. One concludes because $\pi^*$ is the inverse of $\pi_*$ and $\RR \pi_*$.

(2) Let  $f\colon T\to S$ be a morphism of schemes between quasi-compact and quasi-separated schemes and let $ \xymatrix{T\ar[r]^f\ar[d]_{p} & S\ar[d]^{q} \\ X \ar[r]^{\bar f } & Y
}$ be a finite model.   

The diagram  (whose vertical morphisms are isomorphisms by Theorem  \ref{comparison2})
\[  \xymatrix{D_{qc}(T)\ar[d]^{\RR p_*}_{\wr}\ar[r]^{\RR f_*} & D_{qc}(S)\ar[d]^{\RR q_*}_{\wr} \\ D_{qc}(X)\ar[r]^{\RR {\bar f}_*}  & D_{qc}(Y)
}   \] is commutative. This is clear: $\RR q_*\circ\RR f_*=\RR (q\circ f)_* = \RR (\bar f\circ p)_*= \RR {\bar f}_* \circ \RR p_*$.  
\end{rem}

Now, Theorems \ref{comparison}, \ref{comparison2} and Remark \ref{remark} allow to transport the results obtained on finite spaces (Theorems \ref{DQc=D_qc},  \ref{Rqc=R}, \ref{enoughflats} and \ref{SchematicDuality})  to schemes:   

\begin{thm}\label{schemes} Let $S$ be a quasi-compact and semi-separated scheme. Then

(1.a) The natural functor $D\Qcoh (S)\to D_{qc}(S)$ is an equivalence {\rm (\cite{BokstedtNeeman})}.

(1.b) The category $\Qcoh(S)$ has enough flats {\rm (\cite{Murfet})}.

\medskip

Let  $f \colon T \to S$ be a morphism of schemes between quasi-compact and quasi-separated schemes. Then

(2.a) The functor $\RR f_*\colon D_{qc}(T)\to D_{qc} (S)$ has a right adjoint {\rm (\cite{Neeman},\cite{Lipman})}.

(2.b) If $T,S$ are semi-separated, then   the diagram 
\[ \xymatrix{  D\Qcoh(T)  \ar[r]^{\RR_{\text{qc}} f_*}  \ar[d] &  D\Qcoh(S)   \ar[d] \\
D(T) \ar[r]^{\RR f_*} &   D(S) }\]	is conmutative {\rm (\cite{Lipman})} and the functor $\RR_{qc} f_*\colon D\Qcoh (T)\to D\Qcoh(S)$ has a right adjoint {\rm (\cite{Neeman})}.	
\end{thm}

\section{Ringed Spaces}

Let us first recall the notion of an affine ringed space introduced in \cite{Sancho} (see also  \cite{Sanchos}).

\begin{defn} Let $(S,\OO_S)$ be a ringed space and $A=\OO_S(S)$. We say that $(S,\OO_S)$ is an {\it affine} ringed space if:

(1) It is acyclic: $H^i(S,\OO_S)=0$ for any $i>0$.

(2) The functor \[\aligned  \Qcoh(S)&\to \{ A-\text{modules}\}\\ \M&\leadsto \Gamma(S,\M)
\endaligned \] is an equivalence.
\end{defn}

In the topological case (i.e., $\OO_S=\ZZ$), $(S,\ZZ)$ is an  affine ringed space if and only if $S$ is homotopically trivial (where $S$ is assumed to be path-connected, locally path-connected, and  locally simply connected). If $(S,\OO_S)$ is a scheme, then $(S,\OO_S)$ is an affine ringed space if and only if $S$ is an affine scheme, i.e., $S=\Spec A$. See \cite{Sanchos} for the proofs.

\begin{defn} (See \cite[Section. 2.2]{Sancho2}) Let $S$ be a topological space and let $\U=\{U_1,\dots,U_n\}$ be a finite open covering. For each $s\in S$, let us denote $$U^s=\underset{s\in U_i}\bigcap U_i.$$ Let $\sim$ be the equivalence relation on $S$ defined as: $s\sim s'\Leftrightarrow U^s=U^{s'}$, and let $X=S/\sim$ the (finite) quotient set, with the topology associated to the partial order: $[s]\leq [s']$ iff $U^s\supseteq U^{s'}$. The quotient map $\pi\colon X\to S$ is continuous and $\pi^{-1}(U_x)=U^s$ for any $s\in S$, with $x=\pi(s)$. We say that $X$ is the {\it finite topological space associated to $\U$}. If $\OO_S$ is a sheaf of rings on $S$, then $\OO:=\pi_*\OO_S$ is a sheaf of rings on $X$ and $(S,\OO_S)\to (X,\OO)$ is a morphism of ringed spaces. We say that $(X,\OO)$ is the {\it finite ringed space associated to the covering $\U$ of $(S,\OO_S)$}.
\end{defn}

\begin{defn} Let $(S,\OO_S)$ be ringed space and $\U$ a finite open covering. We say that $\U$ is {\it locally affine} if $U^s$ is an affine ringed space for any $s\in S$. This is equivalent to say that the associated map $\pi\colon S\to X$ is ``affine'': $\pi^{-1}(U_x)$ is affine for any $x\in X$. In this case we say that $(X,\OO)$ is a {\it finite model} of $(S,\OO_S)$.
\end{defn}

\begin{rem}\label{rem-qc-qs} It is proved in \cite{Sanchos} that if $\U$ is a locally affine finite covering of $(S,\OO_S)$ and $\pi\colon (S,\OO_S)\to (X,\OO)$ is the associated finite ringed space, then the direct image $\pi_*$ takes quasi-coherent modules on $S$ into quasi-coherent modules on $X$ and the functors
\[ \pi_*\colon \Qcoh(S)\to \Qcoh (X),\qquad  \pi^*\colon \Qcoh(X)\to \Qcoh (S)\] are mutually inverse.
\end{rem}

\begin{defn}\label{qcqs} We say that a ringed space $(S,\OO_S)$ is {\it quasi-compact and quasi-separated} if it admits a locally affine finite covering $\U$ such that for any $U^s\supseteq U^{s'}$ the ring homomorphism $\OO_S(U^s)\to\OO_S(U^{s'})$ is flat. This is equivalent to say that $(S,\OO_S)$ admits a finite model $(X,\OO)$ which is a finite space (Definition \ref{finitespace}).
%
%
\end{defn}

\begin{ejems} \begin{enumerate} \item Any finite simplicial complex $S$ is quasi-compact and quasi-separated (i.e., $(S,\ZZ)$ is a quasi-compact and quasi-separated ringed space). This is due to Mc Cord (\cite{McCord}).
\item Any finite $h$-regular  CW-complex is quasi-compact and quasi-separated (see \cite{BarmakMinian}).
\item If $(S,\OO_S)$ is a scheme, then $(S,\OO_S)$ is a quasi-compact and quasi-separated ringed space if and only if it is a quasi-compact and quasi-separated scheme (see \cite[Proposition 2.4]{Sancho}).
\end{enumerate}
\end{ejems}

\begin{thm}\label{Thm-qcqs} If $(S,\OO_S)$ is a quasi-compact and quasi-separated ringed space, then $\Qcoh(S)$ is a Grothendieck abelian category. Moreover, $\Qcoh(S)$ admits flat covers and cotorsion envelopes.
\end{thm}

\begin{proof} By definition $S$ admits a locally affine finite covering $\U$ such that the associated finite ringed space $(X,\OO)$ is a finite space. Since $\Qcoh(X)$ is a Grothendieck abelian category, one concludes by Remark \ref{rem-qc-qs} that  $\Qcoh(S)$ is also a Grothendieck abelian category. Moreover, the equivalence $\pi^*\colon \Qcoh(X)\to \Qcoh(S)$ preserves tensor products and hence flatness. Since $\Qcoh(X)$ admits flat covers and cotorsion envelopes (\cite{EnochsEstrada}), $\Qcoh(S)$ also does.
\end{proof}

\begin{defn} Let $f\colon (T,\OO_T)\to (S,\OO_S)$ be a morphism of ringed spaces between quasi-compact and quasi-separated ringed spaces. The inverse image $f^*\colon \Qcoh(S)\to\Qcoh(T)$ has a right adjoint, because $\Qcoh(S)$ is a Grothendieck abelian category. This right adjoint shall be denoted by $$f_*^{qc}\colon \Qcoh(T)\to\Qcoh (S)$$ and named {\it quasi-coherent direct image}. If $f_*\colon \Mod(T)\to \Mod(S)$ is the ordinary direct image functor, then for any quasi-coherent module $\M$ on $T$ one has: $f_*^{qc}\M=\Qc(f_*\M)$, where $\Qc\colon \Mod(S)\to \Qcoh(S)$ is the quasi-coherator functor (i.e., the right adjoint of the inclusion functor $\Qcoh(S)\hookrightarrow \Mod(S)$). Since $\Qcoh(T)$ is a Grothendieck abelian category, $f_*^{qc}$ has a right derived functor, which shall be denoted by
\[ \RR_{qc}f_*^{qc}\colon D(\Qcoh(T))\to D(\Qcoh(S).\]
\end{defn}

\noindent{\ \ \bf Notation.} Let $\U$ be a finite open covering of $S$. For any $s,s'\in S$ we shall denote
\[ U^{ss'}=U^s\cap U^{s'}.\]

\begin{defn} Let $\U$ be a finite open covering of a ringed space $(S,\OO_S)$. We say that $\U$ is a {\it semi-separating} covering of $(S,\OO_S)$ if:

(1) $\U$ is locally affine and $\OO_S(U^s)\to\OO_S(U^{t})$ is flat for any $U^s\supseteq U^{t}$ (hence $(S,\OO_S)$ is quasi-compact and quasi-separated).

(2) $ H^i(U^{ss'},\OO_S)=0$ for any $i>0$ and  any $s,s'\in S$.  
  
(3) For any $s,s'\in S$ and any $U^s\supseteq U^t$, the natural morphism
\[ H^0(U^{ss'},\OO_S)\otimes_{\OO_S(U^s)}\OO_S(U^t)\to H^0(U^{ts'},\OO_S)\] is an isomorphism.

A ringed space $(S,\OO_S)$ is called {\it semi-separable} if it admits a semi-separating covering.
\end{defn}

\begin{rem}(1) A locally affine covering $\U$ is semi-separating if and only if the associated finite ringed space $(X,\OO)$ is semi-separated. Thus a semi-separable ringed space is a ringed space that admits a semi-separated finite model.

(2) A scheme $(S,\OO_S)$ is semi-separable if and only if it is a semi-separated scheme.

(3) Any semi-separated finite space is semi-separable, but the converse is not true. For example, a finite topological space is semi-separable if and only if it is homotopically trivial and it is semi-separated if and only if it is irreducible (in particular, it is contractible).
\end{rem}

\begin{thm}[Grothendieck duality for semi-separable ringed spaces]\label{dualityringedspaces} Let $f\colon (T,\OO_T)\to (S,\OO_S)$ be a morphism of ringed spaces between semi-separable ringed spaces. Assume that there exist a semi-separating covering  $\U=\{U_1,\dots, U_n\}$ of $T$ and a semi-separating covering $\V=\{ V_1,\dots,V_m\}$ of $S$ such that:

(1) $f$ is compatible with $\U$ and $\V$: for any $t\in T$ one has that  $U^t\subseteq f^{-1}(V^{f(t)})$.

(2) For any $U^t\supseteq U^{t'}$ and any $V^s\supseteq V^{s'}$ the natural morphisms

\[ \aligned H^i(U^t\cap f^{-1}(V^s),\OO_T)\otimes_{\OO_T(U^t)}\OO_T(U^{t'}) &\to H^i(U^{t'}\cap f^{-1}(V^s),\OO_T) \\ 
H^i(U^t\cap f^{-1}(V^s),\OO_T)\otimes_{\OO_S(V^s)}\OO_S(V^{s'}) &\to H^i(U^{t}\cap f^{-1}(V^{s'}),\OO_T)
\endaligned
\] are isomorphisms for any $i\geq 0$. 

Then  $\RR_{qc} f^{qc}_*\colon D\Qcoh(T)\to D\Qcoh (S)$ has a right adjoint.
\end{thm}

\begin{proof} Let $(X,\OO_X)$ (resp. $(Y,\OO_Y)$) be the finite ringed space associated to $\U$ (resp. to $\V$). They are semi-separated finite spaces, because $\U$ and $\V$ are semi-separating. Condition (1) yields that there exists a continous map $\bar f\colon X\to Y$ such that the diagram
\[\xymatrix{ T\ar[r]^f\ar[d] & S\ar[d]\\ X\ar[r]^{\bar f} & Y
}\] is commutative. Moreover, $\bar f$ is a morphism of ringed spaces. Now, condition (2) implies that $\bar f$ is a schematic morphism (and then ${\bar f}^{qc}_*={\bar f}_*$). Let us consider the diagram
\[\xymatrix{ D\Qcoh(T)\ar[r]^{\RR_{qc}f^{qc}_*}\ar[d]^{\wr} & D\Qcoh(S)\ar[d]^{\wr} \\
D\Qcoh(X)\ar[r]^{\RR_{qc}{\bar f}_*}\ar[d]^{\wr} & D\Qcoh(Y)\ar[d]^{\wr} \\ D_{qc}(X)\ar[r]^{\RR {\bar f}_*}& D_{qc}(Y)
.}\] The vertical functors are equivalences, by Remark \ref{rem-qc-qs} and Theorem \ref{DQc=D_qc}, and the squares are commutative (the higher one is immediate and the lower one  is due to Theorem \ref{Rqc=R}). We conclude  by Theorem \ref{SchematicDuality}. 
\end{proof}

\begin{ejem} Let $f\colon (T,\OO_T)\to (S,\OO_S)$ be a morphism of ringed spaces between semi-separable ringed spaces. Assume that $f$ is ``affine'', i.e., there exists a semi-separating covering $\V=\{ V_1,\dots,V_m\}$ of $S$ such that $f^{-1}(V^s)$ is affine for any $s\in S$. Then  $\RR_{qc} f^{qc}_*\colon D\Qcoh(T)\to D\Qcoh (S)$ has a right adjoint.
\end{ejem}

\end{document}